\documentclass{elsarticle}

\usepackage{amssymb, latexsym, comment, amsmath, amsthm, upgreek, epsf, epsfig, color, verbatim, caption, tabularx, xfrac, float}
\usepackage{amsfonts}
\usepackage{booktabs}
\usepackage{mathtools}
\usepackage{footnote}
\usepackage{url}
\usepackage[paperheight=11in,paperwidth=8.5in,margin=1.25in]{geometry}

\usepackage{array}
\newcolumntype{L}[1]{>{\raggedright\let\newline\\\arraybackslash\hspace{0pt}}m{#1}}
\newcolumntype{C}[1]{>{\centering\let\newline\\\arraybackslash\hspace{0pt}}m{#1}}
\newcolumntype{R}[1]{>{\raggedleft\let\newline\\\arraybackslash\hspace{0pt}}m{#1}}

\makeatletter
\def\ps@pprintTitle{%
\let\@oddhead\@empty
\let\@evenhead\@empty
\def\@oddfoot{\centerline{\thepage}}%
\let\@evenfoot\@oddfoot}
\makeatother

\newtheorem{theorem}{Theorem}[section]
\newtheorem{lemma}[theorem]{Lemma}
\newtheorem{proposition}[theorem]{Proposition}

\begin{document}

\begin{frontmatter}
\title{Besse projective spaces with many diameters}
\author{Ian Adelstein and Franco Vargas Pallete}
\address{Department of Mathematics, Yale University \\ New Haven, CT 06520 United States}

\begin{abstract} It is known that Blaschke manifolds (where injectivity radius equals diameter) are Besse manifolds (where all geodesics are closed). We show that Besse manifolds with sufficiently many diameter realizing directions are Blaschke. We also provide bounds in terms of diameter on the length of the shortest closed geodesic for pinched curvature metrics on simply connected manifolds. 
\end{abstract}
\begin{keyword} closed geodesics
\MSC[2010]  53C22 \sep 53C23
\end{keyword}
\end{frontmatter}

\section{Introduction}

In this paper we study manifolds all of whose geodesics are  closed, the so called Besse manifolds. The standard simply connected examples are the round spheres $S^n$, the complex and quaternionic projective spaces $\mathbb{C}P^{n/2}$ and $\mathbb{H}P^{n/4}$ with their Fubini-Study metric, and the Cayley projective plane $CaP^2$. There are many examples of non-standard Besse metrics on $S^n$ due to Zoll, Berger, Funk, and Weinstein \cite[Chapter 4]{besse}. It is unknown whether $\mathbb{C}P^{n/2}, \mathbb{H}P^{n/4}$, or $CaP^2$ admit non-standard Besse metrics. The real projective spaces $\mathbb{R}P^n$ do not admit non-standard Besse metrics \cite{lin}. 

A theorem of Bott-Samuelson \cite[Theorem 7.2]{besse} states that any simply connected Besse manifold has the integral cohomology ring of $S^n, \mathbb{C}P^{n/2}, \mathbb{H}P^{n/4}$, or $CaP^2$, its model space. Each of these spaces endowed with its standard metric is an example of a Blaschke manifold (one where injectivity radius equals diameter). It is known that Blaschke manifolds are Besse, and is conjectured that any Blaschke manifold is isometric to its model space; the conjecture is only resolved for spheres and real projective spaces \cite[Appendix D]{besse}. 

It is therefore interesting to consider when a Besse manifold is Blaschke. 
Consider a Besse metric on a homotopy $n$-sphere where every prime geodesic has length twice the diameter. If every point admits a diameter realizing direction then the manifold is Blaschke  \cite[Proposition 3.1]{sch}, and therefore round by the resolved Blaschke conjecture for spheres. 

Our main result demonstrates that a Besse manifold modeled on a projective space with sufficiently large smooth families of diameter realizing directions is Blaschke. Here we denote by $C(p) \subset M$ the cut locus of $p\in M$.

\begin{theorem}\label{main1}
Let $M^n$ be a Besse manifold with diameter one in which all prime geodesics have length two. Assume $M$ is homotopy equivalent to $\mathbb{C}P^{n/2},~ \mathbb{H}P^{n/4}$ or $CaP^2$. For every $p \in M$ assume there is a smooth closed submanifold $N^{n-k} \subseteq C(p)$ with $d(N,p)=1$ where $k = 2, 4$ or $8$, respectively. Then $N=C(p)$ and the manifold is Blaschke. 
\end{theorem}

We make some notes about our assumptions. First, a conjecture of Berger states that simply connected Besse manifolds have all prime geodesics of the same length. This has been resolved for $n$-spheres with $n = 2$ \cite{GG} and $n \geq 4$ \cite{RW}, and we assume here that all prime geodesics have the same length. Second, we do not require the full homotopy equivalence, only that the homotopy groups agree in dimensions $k$ and below, for $k = 2, 4$ or $8$, respectively. Note that Bott-Samuleson already guarantees the Besse manifold has the integral cohomology ring of $\mathbb{C}P^{n/2},~ \mathbb{H}P^{n/4}$ or $CaP^2$, and indeed goes further for the case of $\mathbb{C}P^{n/2}$ to guarantee the homotopy type. Finally, the assumption that every $p \in M$ admits a smooth $N^{n-k} \subseteq C(p)$ with $d(N,p)=1$ is precisely what we mean when we say the Besse manifold has sufficiently many diameter realizing directions. The standard metrics on the models all have such an $N^{n-k} = C(p)$. A priori our Besse manifolds could admit this $N^{n-k}$ as a proper subset of $C(p)$; we prove $N^{n-k} = C(p)$ and conclude that injectivty radius must equal diameter. 

As a brief proof sketch, for $p,q \in M$ with $d(p,q)=1$ we use that fact that $\pi_{k-1}(\Lambda(p,q))=\pi_{k} (M)=\mathbb{Z}$ to pick a nontrivial representative $\hat{g} \colon S^{k-1} \to \Lambda(p,q)$. The Besse condition (via some Morse theory) allows us to homotope this map to $g \colon S^{k-1} \to \Lambda^1(p,q)$ where $\Lambda^1(p,q)$ is the space of paths with energy less than or equal to one. We therefore have that diameter realizing points occur in $S^{k-1}$ families, and show that $T_qM$ can be orthogonally decomposed into vectors tangent to $N^{n-k}$ and initial velocity vectors of minimizing geodesics from $q$ to $p$. We then show that $N^{n-k} = C(p)$ and can conclude that injectivity radius equals diameter.

As a bridge to our next result, we note in the absence of the Besse assumption that bounds on sectional curvature have been applied when studying this class of manifolds. In \cite{SSW} it is shown that a simply connected manifold with ${\bf sec} \leq 1$ and every geodesic having a conjugate point at $t=\pi$ is isometric to $ S^n,~ \mathbb{C}P^{n/2},~ \mathbb{H}P^{n/4}$ or $CaP^2$. The proof first shows that such manifolds are Blaschke, and then applies a special case of the Blaschke conjecture. The assumption that \emph{every} geodesic has a conjugate point at $t=\pi$ is more restrictive than our assumption that a codimension $k$ subset of geodesics have a cut point at distance diameter, but the conclusion that the manifold is isometric to $ S^n,~ \mathbb{C}P^{n/2},~ \mathbb{H}P^{n/4}$ or $CaP^2$ is stronger than our conclusion that the manifold is Blaschke (and thus only conjecturally isometric to one of these space). 

A question that many authors have asked \cite{croke1, maeda, NR2002, Sab, Ade} is whether there exist constants $c(n)$ such that the length of the shortest closed geodesic $L(M^n)$ on a closed Riemannian manifold $M^n$ is bounded above by $c(n) D(M^n)$, where $D(M^n)$ is the diameter of the manifold. 
This question has a quick answer for non-simply connected manifolds:~the shortest non-contractible closed curve is a geodesic with length bounded above by $2D(M^n)$. 
Here we address this question for simply connected manifolds $M^n$ with metrics of pinched sectional curvature. We combine an upper bound on the length of the shortest closed geodesic in the pinched curvature setting due to Ballmann, Thorbergsson and Ziller \cite[Theorem 1.4]{btz1} with Klingenberg's lower bounds on injectivity radius and therefore diameter to yield the following:

\begin{theorem}\label{main} Let $M^n$ be a simply connected compact Riemannian manifold with sectional curvature $K$ satisfying $0< \delta \leq K \leq 1$. Then 

\begin{enumerate}
\item $L(M^n, g) \leq \frac{2}{\sqrt{\delta}} D(M^n, g) $ 
\item $L(M^n, g) \leq \frac{1}{\sqrt{\delta}} D(M^n, g)$ when $M^n$ is not homotopy equivalent to $S^n$.
\end{enumerate}
In either case, equality implies that $M^n$ is isometric to a symmetric space.
\end{theorem}

We note that the bound $L(S^n, g) \leq \frac{2}{\sqrt{\delta}} D(S^n, g) $ is also achieved in \cite{Ade} via different techniques, albeit in the more restrictive setting of pinching constant $0.83 < \delta \leq K \leq 1$.

The simply connected manifolds $\mathbb{C}P^{n/2}, ~\mathbb{H}P^{n/4}$ and $CaP^2$ all admit metrics of positive sectional curvature. 
Our result applies to these metrics, and to the best of our knowledge this is the first diameter bound on the length of the shortest closed geodesic for this class of spaces. Moreover, these spaces (together with the sphere) are the only simply connected symmetric spaces admitting metrics of positive curvature. Therefore in the equality setting of the theorem we have that $M^n$ is isometric to one of $S^n, \mathbb{C}P^{n/2}, ~\mathbb{H}P^{n/4}$ or $CaP^2$.

This result is sharp in the following sense. 
In the first case, as the pinching constant $\delta$ approaches 1, the constant $c(n)$ approaches $2$, as would be expected by the limiting case of the round sphere. In the second case, when $M^n$ is not homotopy equivalent to $S^n$, we have that $\delta \leq 1/4$ by the quarter-pinched sphere theorem. As the pinching constant $\delta$ approaches $1/4$ the constant $c(n)$ approaches $2$, as would be expected by the limiting cases of the standard metrics on $\mathbb{C}P^{n/2}, ~\mathbb{H}P^{n/4}$ and $CaP^2$.

The paper proceeds as follows. In Section 2 we introduce the Morse theory that allows us to conclude that diameter realizing directions occur in $S^{k-1}$ families. The arguments follow closely those in \cite{sch} and \cite{lin}. In Section 3 we study the relationship between $N^{n-k} \subseteq C(p)$ and the minimizing geodesics back to $p$, concluding that $N^{n-k} = C(p)$ and finishing the proof of Theorem~\ref{main1}. Section 4 provides additional background on the diameter bound question and contains the proof of Theorem~\ref{main}.

\vspace{.1in}

\textbf{Acknowledgements:}~The authors would like to thank Wolfgang Ziller for suggesting the reference \cite{btz1} and helpful discussions about short closed geodeiscs in the pinched curvature setting. Vargas Pallete was supported by DMS-2001997.

\section{Morse Theory}

The goal of this section is to use Morse theory to show that diameter realizing points on Besse metrics on $ \mathbb{C}P^{n/2}, \mathbb{H}P^{n/4}$ or $CaP^2$ occur in $S^{k-1}$ families, where $k = 2, 4$ or $8$, respectively. This result is an extension of \cite[Proposition 3.1]{sch} which says that diameter realizing points on Besse metrics on $S^n$ occur in $S^{n-1}$ families, and therefore if $p \in S^n$ is a diameter realizing point its cut locus $C(p)$ is a singleton at distance diameter. 

\begin{proposition}\label{prop:morse}
Let $M^n$ be a Besse manifold with diameter one in which all prime geodesics have length two. Assume $M$ is homotopy equivalent to $ \mathbb{C}P^{n/2},~ \mathbb{H}P^{n/4}$ or $CaP^2$. If $d(p,q)=1$ then there exists a continuous map $g \colon S^{k-1} \to \Lambda^1(p,q)$ representing a nontrivial class in $\pi_{k-1}(\Lambda(p,q))=\pi_{k} (M)=\mathbb{Z}$ with $k = 2, 4$ or $8$, respectively. Note that the paths in this $S^{k-1}$ family are minimizing geodesics of length diameter between $p$ and $q$. 
\end{proposition}

We review some well known preliminaries before proving the proposition. Given $(p,q) \in M \times M$, let $\Lambda(p,q)$ denote the space of piecewise smooth paths $c:[0,1]\rightarrow M$ with $c(0)=p$ and $c(1)=q$.  Define the length and energy functions $$L: \Lambda(p,q) \rightarrow [0,\infty)\,\,\,\,\,\text{and}\,\,\,\,\,E:\Lambda(p,q) \rightarrow [0,\infty)$$  by $$L(c)=\int_0^1 ||\dot{c}(t)||dt\,\,\,\,\, \text{and}\,\,\,\,\, E(c)=\int_0^1 ||\dot{c}(t)||^2dt.$$ The Cauchy-Schwartz inequality implies $$L^2(c) \leq E(c)$$ with equality holding if and only if $c$ has constant speed $||\dot{c}(t)||$.  Given $e \in [0,\infty)$, let $$\Lambda_e(p,q)=E^{-1}([0,e))\,\,\,\,\, \text{and}\,\,\,\, \Lambda^e(p,q)=E^{-1}([0,e]).$$

A path is a critical point for $E$ if and only if the path is a geodesic.  The index of a critical point $\gamma \in \Lambda(p,q)$ equals the number of parameters $s \in (0,1)$ for which $\gamma(s)$ is conjugate to $\gamma(0)$ along $\gamma$, counted with multiplicities.

\begin{proof}
Because $M$ is homotopy equivalent to $ \mathbb{C}P^{n/2},~ \mathbb{H}P^{n/4}$ or $CaP^2$ we have that $\pi_{k-1}(\Lambda(p,q))=\pi_{k} (M)=\mathbb{Z}$ for $k = 2, 4$ or $8$, respectively, and can choose a nontrivial continuous representative $\hat{g} \colon S^{k-1} \to \Lambda(p,q)$. The idea is to homotope $\hat{g}$ to the desired map $g \colon S^{k-1} \to \Lambda^1(p,q)$. This homotopy relies on a number of previously established lemmas:~the fact that geodesic segments on $M^n$ of energy greater than four have index at least $n$ \cite[Lemma 3.2]{sch} and \cite[Lemma 2.4]{lin}; the fact that we can homotope through critical points of index at least $k$  \cite[Theorem 2.5.16]{kling} as applied in the proofs of \cite[Proposition 3.1]{sch} and \cite[Theorem 1]{lin}; and the fact that the critical values of $E \colon \Lambda(p,q) \to \mathbb{R}$ are precisely the squares of odd integers when $d(p,q)=1$ \cite[Lemma 2.2]{lin}. 

The proof requires some extra work to address the possibility that $p$ and $q$ may be conjugate along some geodesic. Choose $z\in M$ such that $p$ is not conjugate to $z$ along any geodesic and $\epsilon \vcentcolon = d(q,z) < \min \{1/2, \text{inj}(q) \}$. This choice ensures that critical points of $E \colon \Lambda(p,z) \to \mathbb{R}$ are nondegenerate. 
Let $\hat{f} \colon S^{k-1} \to \Lambda(p,z)$ be a nontrivial representative of $\pi_{k-1} (\Lambda(p,z)) = \pi_k (M)= \mathbb{Z}$. 
In the absence of degenerate critical points (and using the fact that geodesic segments of energy greater than four have index at least $n$) we homotope $\hat{f}$ through these high index critical points to an intermediate map  $f \colon S^{k-1} \to \Lambda^{4.1}(p,z)$, c.f.~\cite[Theorem 2.5.16]{kling}.

For $c \in \Lambda(p,z)$ and $\tau:[0,\epsilon] \rightarrow M$ a minimizing geodesic with $\tau(0)=z$ and $\tau(\epsilon)=q$ we define $$\tau* c\in \Lambda(p,q)$$ by  \[ 
\tau* c(t)=
\begin{cases}
c(\frac{t}{1-\epsilon}) & \text{for } t\in [0,1-\epsilon]\\
\tau(t-(1-\epsilon)) & \text{for } t\in[1-\epsilon,1]\\
\end{cases}
\]  so that $$E(\tau*c)=\frac{E(c)}{1-\epsilon}+\epsilon.$$ Given a map $f:S^{k-1} \rightarrow \Lambda(p,z)$, the map $\tau f:S^{k-1}\rightarrow \Lambda(p,q)$ defined by $\tau f(\theta)=\tau*f(\theta)$ for each $\theta \in S^{n-1}$, represents a nontrivial homotopy class of maps if and only if $f$ represents a nontrivial homotopy class of maps.

Finally define $\hat{g}:S^{k-1} \rightarrow \Lambda(p,q)$ by $\hat{g}=\tau f.$ For each $\theta \in S^{k-1}$, $$E(\hat{g}(\theta))=\frac{E(f(\theta))}{1-\epsilon}+\epsilon\leq 2E(f(\theta))+\frac{1}{2}\leq 8.7<9.$$ To conclude the proof we use the fact that the critical values of $E \colon \Lambda(p,q) \to \mathbb{R}$ are precisely the squares of odd integers when $d(p,q)=1$ \cite[Lemma 2.2]{lin}, and homotope $\hat{g}  \colon S^{k-1} \to \Lambda^{8.7}(p,q)$ through the critical value free interval $(1,9)$ to the desired map $g \colon S^{k-1} \to \Lambda^1(p,q)$. 
\end{proof}

\section{Decomposing the Tangent Space}

In this section we use results from differential topology to show how our assumptions, together with the Morse theoretic result, can be used to prove our theorem. 

Assume $M$ is as above so that there exists a continuous map $g \colon S^{k-1} \to \Lambda^1(p,q)$ representing a nontrivial class in $\pi_{k-1}(\Lambda(p,q))=\pi_{k} (M)=\mathbb{Z}$ with $k = 2, 4$ or $8$, respectively. An important lemma is as follows:

\begin{lemma}
Let $p,q \in M^n$ be diameter realizing points so that $q \in C(p)$. Assume there is a smooth subset $N^{n-k} \subseteq C(p)$ with $d(N,p)=1$ and $q \in N$. Then every geodesic with starting velocity normal to $N$ at $q$ is a minimizing geodesic from $q$ to $p$.
\end{lemma}

\begin{proof}

Take a smooth path $q(t)$ in $N$ and consider $\gamma$ a minimizing geodesic between $p$ and $q=q(0)$. Consider $\gamma_t$ a smooth $1$-parameter family of paths so that the endpoints of $\gamma_t$ are p and $q(t)$. Since $q(t)\in N$ it follows that $\ell(\gamma_t)\geq 1$, and by assumption $\ell(\gamma_0)=1$. Then by the first variational formula of length it follows that $\gamma'$ at $q$ is orthogonal to $q'(0)$.

Denoting by $UT^{\perp}_qN$ to the unit orthogonal vectors to $N$ at $q$ in $M$, then from the previous paragraph the map $g\colon S^{k-1} \to \Lambda^1(p,q)$ factors through the exponential map as a continuous map $h:S^{k-1}\rightarrow UT^{\perp}_q N$. The lemma follows from showing that $h$ is surjective. We will first see that there exists a continuous map $G\colon UT^{\perp}_qN \rightarrow \Lambda(p,q)$ so that $G\circ h =g$.

$G$ is already defined on $h(S^{k-1})$ by the exponential map. Given $\epsilon$ there exists an open neighbourhood $U \subseteq UT^{\perp}_qN$ of $h(S^{k-1})$ so that for any $u\in U$ the unit length geodesic $\gamma_u$ with starting velocity $u$ ends at distance less than $\epsilon$ from p. By taking $\epsilon$ small and concatenating $\gamma_u$ with the minimizing segment between $\gamma_u(1)$ and $p$, we can extend $G$ to be defined on the open set $U$.

Consider now $K \subset U$ compact manifold with boundary so that $K \supset h(S^{k-1})$, and fix a handle decomposition of $UT^{\perp}_qN \setminus K$. We will extend $G$ inductively over the handle decomposition of $UT^{\perp}_qN \setminus K$. For each $j$-handle $D^j\times D^{k-j-1}$ ($0\leq j\leq k-2$) we have $G$ already defined on the boundary component $S^{j-1}\times D^{k-j-1}$. By contracting in the $D^{k-j-1}$ direction, we can assume that $G$ in $S^{j-1}\times D^{k-j-1}$ only depends on the first coordinate. Hence extending $G$ reduces to knowing whether $G(S^{j-1})$ is a trivial homotopy class in $\Lambda^1(p,q)$. But since $\pi_{j-1}(\Lambda^1(p,q)) \simeq \pi_j(M)$, then for $0\leq j\leq k-1$ we have that $G(S^{j-1})$ is homotopically trivial since by assumption the homotopy groups themselves are trivial.

Having defined $G$ we are ready to see that $h$ is surjective. If this is not the case, there there is a point $p$ in $UT_q^\perp N$ not in the image of $h$. By composing a point retraction of $UT_q^\perp N\setminus\lbrace p\rbrace$ between $G$ and $h$, we will have that $g:S^{k-1}\rightarrow\Lambda(p,q)$ is homotopic to a constant map, which contradicts the fact that $g$ represented a non-trivial homotopy class.
\end{proof}

We are now ready to prove Theorem~\ref{main1}. Given any $q\in N \subseteq C(p)$ and $v\in UT^\perp_q N$, we know by the lemma that ${\rm exp}(tv)$ is a minimizing geodesic from $q$ to $p$, so that $\lbrace {\rm exp}(tv)\,|\, 0\leq t<1 \rbrace \cap C(p) =\lbrace q\rbrace$. Denoting the normal vectors to $N$ of norm less than $1$ by $T^\perp_1 N$, we have that the exponential map ${\rm exp} \colon T^\perp_1 N \rightarrow M$ satisfies ${\rm exp}(T^\perp_1 N) \cap C(p) = N$. Since $D({\rm exp})$ is an isomorphism at the zero section of $T^\perp_1 N$, then the set ${\rm exp}(T^\perp_1 N)$
is open around $N$ by the Inverse Function Theorem. Since the cut locus $C(p)$ is connected (c.f.~Chapter 13, Section 2 of \cite{DoCarmo}) and ${\rm exp}(T^\perp_1 N) \cap C(p) = N$, it follows that $C(p)=N$ and the theorem is proved.

\section{Pinched Curvature}

Here we address the question as to whether there exist constants $c(n)$ such that the length of the shortest closed geodesic $L(M^n)$ on a closed Riemannian manifold $M^n$ is bounded above by $c(n) D(M^n)$, where $D(M^n)$ is the diameter of the manifold. 

In addition to the quick bound of $2D(M^n)$ for non-simply connected manifolds, curvature free bounds on the length of the shortest closed geodesic in terms of diameter have only been given for manifolds homeomorphic to the 2-sphere. Croke \cite{croke1} provided the first such bound, proving that $L(S^2,g) \leq 9D(S^2,g)$. Maeda \cite{maeda} improved Croke's techniques to achieve the bound $L(S^2,g) \leq 5D(S^2,g)$. Nabutovsky and Rotman \cite{NR2002} and independently Sabourau \cite{Sab} developed new techniques to prove that $L(S^2,g) \leq 4D(S^2,g)$. By imposing bounds on curvature, the authors have previously shown in \cite{Ade} that $L(S^2,g) \leq 3D(S^2,g)$ for non-negative metrics on the 2-sphere and that $L(S^2,g) \leq \frac{2}{\sqrt{\delta}}D(S^2,g)$ for $\delta>.83$ pinched metrics on the 2-sphere.

One might conjecture that $L(M^n) \leq 2D(M^n)$ for all closed Riemannian manifolds $M^n$. This conjecture turns out to be overly optimistic, even for the case of the 2-sphere. Balacheff, Croke, and Katz \cite{bck} demonstrated the existence of Zoll spheres with $L(S^2, Zoll) > 2D(S^2, Zoll)$. Recall that a Zoll sphere is a metric on the 2-sphere all of whose geodesics are closed and of the same length. These examples are not constructive, and one can not say how much longer than $2D(M^n, g)$ the shortest closed geodesic might be.

In addition to addressing this question for simply connected manifolds with metrics of pinched curvature, our Theorem \ref{main} offers insight into the Balacheff, Croke, and Katz \cite{bck} examples cited previously. For these examples we can now add an upper bound on the length of the shortest closed geodesic: $$2D(S^2, Zoll) < L(S^2, Zoll) < \frac{2}{\sqrt{\delta}} D(S^2, Zoll). $$

The proof of Theorem \ref{main} follows rather quickly by combining an upper bound on the length of the shortest closed geodesic in the pinched curvature setting due to Ballmann, Thorbergsson and Ziller \cite[Theorem 1.4]{btz1} with Klingenberg's lower bound on injectivity radius and therefore diameter. 

\begin{theorem}[\cite{btz1}, Theorem 1.4]
If the sectional curvature of $M^n$ satisfies $0< \delta \leq K \leq 1$, then there exists a closed geodesic with length $\leq \frac{2\pi}{\sqrt{\delta}}$. If $M^n$ is not homotopy equivalent to $S^n$ then length $\leq \frac{\pi}{\sqrt{\delta}}$.
\end{theorem}

\begin{lemma}[Klingenberg]
If the sectional curvature of $M^n$ satisfies $0< \delta \leq K \leq 1$, then either
\begin{enumerate}
    \item $i(M^n) \geq \pi$, or
    \item there exists a closed geodesic with length equal to $2i(M^n)$
    \end{enumerate}
    where $i(M^n)$ denotes the injectivity radius of $M^n$. 
\end{lemma}

\begin{proof}[Proof of Theorem \ref{main}]

rem can only occur if $D(M^n) = i(M^n) = \pi$ which is the Blaschke condition.

We first establish the inequalities. Applying Klingenberg's Lemma we consider the two cases.
In the first case the bounds (both in terms of $\pi$) on $D(M^n)$ and $L(M^n)$ allow us to relate these quantities. We combine the inequality $D(M^n) \geq i(M^n) \geq \pi$ with the bounds on length from \cite[Theorem 1.4]{btz1} to yield our theorem.
The second case is even quicker, as Klingenberg already provides the bound $L(M^n)=2i(M^n) \leq 2D(M^n)$. When $M^n$ is not homotopy equivalent to $S^n$ the quarter-pinched sphere theorem implies $\delta \leq 1/4$, and we conclude that $2 \leq 1/\sqrt{\delta}$ and $L(M^n, g) \leq \sfrac{1}{\sqrt{\delta}} D(M^n, g)$.

We now focus on the equality setting. Our first goal is to establish (in both cases of Klingenberg's Lemma) that $M^n$ is a Blaschke manifold (i.e.~that injectivity radius equals diameter). 
In the first case, when $D(M^n) \geq i(M^n) \geq \pi$, we see that equality in our theorem can only occur if $D(M^n) = i(M^n) = \pi$ which is the Blaschke condition. 
In the second case, we combine the inequality $L(M^n)=2i(M^n) \leq 2D(M^n)$ with the fact that $2D(M^n) \leq L(M^n)$ (which follows from equality in our theorem and the respective upper bounds on $\delta$) to again yield the Blaschke condition. 

The final step of the equality proof is to note that for Blaschke manifolds all geodesics are periodic and of length twice the diameter. Thus when $M^n$ is homotopy equivalent to $S^n$ we have $\delta=1$ and otherwise we have $\delta=1/4$. The equality result then follows from rigidity of the sphere when $\delta=1$ and Berger's minimal diameter theorem \cite[Theorem 6.6]{Che} when $\delta=1/4$.
\end{proof}

\end{document}